\documentclass[a4paper]{amsart} 
\usepackage{amssymb} 
\usepackage{stmaryrd}
\usepackage{tikz}
\usepackage{tikz-cd}
\usepackage{hyperref}

\title{Length categories of infinite height}
\author{Henning Krause}
\address{Henning Krause\\ Fakult\"at f\"ur Mathematik\\
  Universit\"at Bielefeld\\ D-33501 Bielefeld\\ Germany.}
\email{hkrause@math.uni-bielefeld.de}
\author{Dieter Vossieck}
\address{Dieter Vossieck\\ Fakult\"at f\"ur Mathematik\\
  Universit\"at Bielefeld\\ D-33501 Bielefeld\\ Germany.}
\email{vossieck@math.uni-bielefeld.de}

\theoremstyle{plain}
\newtheorem{thm}{Theorem}[section]
\newtheorem{prop}[thm]{Proposition}
\newtheorem{lem}[thm]{Lemma} 

\newtheorem{cor}[thm]{Corollary}

\theoremstyle{definition}

\newtheorem{exm}[thm]{Example}

\theoremstyle{remark}
\newtheorem{rem}[thm]{Remark}

\numberwithin{equation}{section}

\newcommand{\Coker}{\operatorname{Coker}}
\newcommand{\colim}{\operatorname*{colim}}

\renewcommand{\dim}{\operatorname{dim}}

\newcommand{\Eff}{\operatorname{Eff}}
\newcommand{\End}{\operatorname{End}}

\newcommand{\Ext}{\operatorname{Ext}}

\newcommand{\fp}{\operatorname{fp}}
\newcommand{\Fp}{\operatorname{Fp}}

\newcommand{\height}{\operatorname{ht}}

\newcommand{\Hom}{\operatorname{Hom}}

\newcommand{\ind}{\operatorname{ind}}
\newcommand{\Ind}{\operatorname{Ind}}

\newcommand{\Ker}{\operatorname{Ker}}

\newcommand{\Lex}{\operatorname{Lex}}

\renewcommand{\mod}{\operatorname{mod}}

\newcommand{\oHom}{\operatorname{\overline{Hom}}}

\newcommand{\rad}{\operatorname{rad}}

\newcommand{\soc}{\operatorname{soc}}

\newcommand{\St}{\operatorname{St}}

\newcommand{\supp}{\operatorname{supp}}

\newcommand{\uHom}{\operatorname{\underline{Hom}}}

\newcommand{\Ab}{\mathrm{Ab}}

\newcommand{\op}{\mathrm{op}}

\newcommand{\comp}{\mathop{\circ}}

\newcommand{\lto}{\longrightarrow}

\newcommand{\xto}{\xrightarrow}

\def\A{\mathcal A} 
 
\def\C{\mathcal C}

\def\X{\mathcal X}
\def\Y{\mathcal Y}

\def\bbZ{\mathbb Z}

\newcommand{\frakm}{\mathfrak{m}} 
\newcommand{\frako}{\mathfrak{0}}

\def\a{\alpha}
\def\b{\beta}

\def\p{\phi}

\def\t{\tau}

\def\La{\Lambda}

\begin{document}

\dedicatory{Dedicated to Dave Benson on the occasion of his 60th birthday.}
\thanks{February 17, 2017.}
\begin{abstract}
  For abelian length categories the borderline between finite and
  infinite representation type is discussed. Characterisations of
  finite representation type are extended to length categories of
  infinite height, and the minimal length categories of infinite
  height are described.
\end{abstract}


\maketitle

\setcounter{tocdepth}{1}
\tableofcontents

\section{Introduction}

An abelian category is a \emph{length category} if it is essentially
small and every object has a finite composition series \cite{Ga1973}.
The \emph{height} of a length category is the supremum of the Loewy
lengths of all objects.

The aim of this note is to explore the structure of length categories
of infinite height. Length categories of finite height arise from
artinian rings by taking the category of finite length modules. Also,
length categories of infinite height are ubiquitous, and typical
examples are the uniserial categories which are not of finite
height. Recall that a length category is \emph{uniserial} if every
indecomposable object has a unique composition series
\cite{AR1968}. For instance, the category of nilpotent finite
dimensional representations of a cyclic quiver over any field is
uniserial and of infinite height.

The paper is divided into three parts. First we extend known
characterisations of finite representation type for module categories
to more general length categories, including those of infinite height
(Theorem~\ref{th:indec}). Then we show that uniserial categories
satisfy these finiteness conditions (Corollary~\ref{co:uniserial} and
Theorem~\ref{th:uniserial}). In the final part, we describe the
minimal length categories of infinite height, and it turns out that
only uniserial categories occur (Theorem~\ref{th:minimal}).

\section{Length categories}

In this section we collect some basic concepts that are relevant for
the study of  abelian length categories.

For a module $M$ over a ring $\La$ let $\ell_\La(M)$ denote its
composition length.

\subsection*{Ext-finite categories}

A length category $\C$ is
\emph{(left) Ext-finite} if for every pair of simple objects $S$ and $T$
\[\ell_{\End_\C(T)}(\Ext_\C^1(S,T))<\infty.\] 
A length category $\C$ is equivalent to a module category (consisting
of the finitely generated modules over a right artinian ring) if and
only if the following holds \cite{Ga1973}:
\begin{enumerate}
\item The category $\C$ has only finitely many simple objects.
\item The category $\C$ is Ext-finite.
\item The supremum of the Loewy lengths of the objects in $\C$ is
  finite.
\end{enumerate}

\subsection*{Hom-finite categories}
Let $\C$ be an essentially small additive category. Let us call $\C$
\emph{(left) Hom-finite} if for all objects $X,Y$ in $\C$ the
$\End_\C(Y)$-module $\Hom_\C(X,Y)$ has finite length. Clearly, this
property implies that $\C$ is a Krull-Schmidt category, assuming that
$\C$ is idempotent complete.

\begin{lem}\label{le:Hom-finite}
Let $\C$ be a Krull-Schmidt category. Then $\C$ is Hom-finite provided
that for all pairs of indecomposable objects $X,Y$ the
$\End_\C(Y)$-module  $\Hom_\C(X,Y)$ has finite length. 
\end{lem}
\begin{proof}
  Choose decompositions $X=\bigoplus_i X_i$ and
  $Y=\bigoplus_j Y_j^{n_j}$, $n_j>0$, such that the $Y_j$ are
  indecomposable and pairwise non-isomorphic. Set
  $Y'=\bigoplus_j Y_j$. Then
\[\ell_{\End_\C(Y)}(\Hom_\C(X,Y))=\sum_i\ell_{\End_\C(Y)}(\Hom_\C(X_i,Y))\]
and 
\begin{align*}\ell_{\End_\C(Y)}(\Hom_\C(X,Y))&=\ell_{\End_\C(Y')}(\Hom_\C(X,Y'))\\
&=\sum_j \ell_{\End_\C(Y_j)}(\Hom_\C(X,Y_j))
\end{align*}
since
\[\End_\C(Y')/\rad \End_\C(Y')\cong\prod_j \End_\C(Y_j)/\rad \End_\C(Y_j).\]
Now the assertion follows.
\end{proof}

\begin{exm}
  Let $k$ be a commutative ring and $\C$ a $k$-linear category.  If
  the $k$-module $\Hom_\C(X,Y)$ has finite length for all $X,Y$ in
  $\C$, then $\C$ and $\C^\op$ are Hom-finite.
\end{exm}

\subsection*{Finitely presented and effaceable functors}

Let $\C$ be an abelian category. 
An additive functor $F\colon\C\to\Ab$ is \emph{finitely presented} if there
is a presentation
\begin{equation}\label{eq:fp}
\Hom_\C(Y,-)\lto \Hom_\C(X,-)\lto F\lto 0,
\end{equation}
and we call $F$ \emph{effaceable} if there is such a presentation such
that the morphism $X\to Y$ in $\C$ is a monomorphism.  We denote by
$\Fp(\C,\Ab)$ the category of finitely presented functors
$F\colon\C\to\Ab$ and by $\Eff(\C,\Ab)$ the full subcategory of
effaceable functors. Note that $\Fp(\C,\Ab)$ is an abelian  category,
and  $\Eff(\C,\Ab)$ is a Serre subcategory.

We recall the following duality.  The assignment
$F\mapsto F^{\scriptscriptstyle\vee}$ given by
\[ F^{\scriptscriptstyle\vee}(X)=\Ext^2(F,\Hom_\C(X,-))\] yields an equivalence
\begin{equation}\label{eq:eff}
\Eff(\C,\Ab)^\op\xto{\ \sim\ }\Eff(\C^\op,\Ab),
\end{equation}
where $\Ext^2(-,-)$ is computed in the abelian category $\Fp(\C,\Ab)$;
see Theorem~3.4 in Chap.~II of \cite{Au1978}.  If
$0\to X\xto{\a} Y\xto{\b} Z\to 0$ is an exact sequence in $\C$ and
$F=\Coker\Hom_\C(\a,-)$, then
$F^{\scriptscriptstyle\vee}\cong \Coker\Hom_\C(-,\b)$ and
$F^{\scriptscriptstyle\vee\scriptscriptstyle\vee}\cong F$.

The \emph{Yoneda functor}
\[\C\lto\Fp(\C^\op,\Ab),\quad X\mapsto \Hom_\C(-,X)\]
admits an exact left adjoint that sends $\Hom_\C(-,X)$ to $X$; it
annihilates the effaceable functors and induces an exact functor
\begin{equation}\label{eq:eff-formula}
\frac{\Fp(\C^\op,\Ab)}{\Eff(\C^\op,\Ab)}\lto \C
\end{equation}
which is an equivalence; see \cite[p.~205]{Au1966} and \cite[III, Prop.~5]{Ga1962}.

\section{Grothendieck groups and almost split sequences}

Let $\C$ be an essentially small abelian category. The
\emph{Grothendieck group} $K_0(\C)$ is the abelian group generated by
the isomorphism classes $[C]$ of objects $C\in\C$ subject to the
relations $[C']-[C]+[C'']$, one for each exact sequence
$0\to C'\to C\to C''\to 0$ in $\C$. Analogously, we write $K_0(\C,0)$
for the abelian group generated by the isomorphism classes $[C]$ of
objects $C\in\C$ subject to the relations $[C']-[C]+[C'']$, one for
each split exact sequence $0\to C'\to C\to C''\to 0$ in $\C$. Thus
there is a canonical epimorphism
\[\pi\colon K_0(\C,0)\lto K_0(\C).\]

Our aim is to find out when the kernel of $\pi$ is generated by
elements $[X]-[Y]+[Z]$ that are given by almost split sequences
$0\to X\to Y\to Z\to 0$ in $\C$.

\subsection*{Almost split sequences}

Let $\C$ be a Krull-Schmidt category. Recall from \cite{AR1975} that a
morphism $\a\colon X\to Y$ in $\C$ is \emph{left almost split} if it is
not a split mono and every morphism $X\to Y'$ that is not a split mono
factors through $\a$.  Dually, a morphism $\b\colon Y\to Z$ in $\C$ is
\emph{right almost split} if it is not a split epi and every morphism
$Y'\to Z$ that is not a split epi factors through $\b$. An exact
sequence $0\to X\xto{\a} Y\xto{\b} Z\to 0$ is \emph{almost split} if
$\a$ is left almost split and $\b$ is right almost split.

We say that $\C$ \emph{has almost split sequences} if for every
indecomposable object $X\in\C$, there is an almost split sequence
starting at $X$ when $X$ is non-injective, and there is an almost
split sequence ending at $X$ when $X$ is non-projective.

\begin{lem}\label{le:almost-split}
A morphism $\a\colon X\to Y$ in $\C$ is left almost split if and
  only $X$ is indecomposable and $\a$ induces an exact sequence
\[\Hom_\C(Y,-)\lto \Hom_\C(X,-)\lto F\lto 0\] in $\Fp(\C,\Ab)$ such
that $F$ is a simple object.
\end{lem}
\begin{proof}
 See Proposition~2.4 in Chap.~II of \cite{Au1978}.
\end{proof}

\begin{lem}\label{le:almost-split2}
For an indecomposable object $X\in\C$ are equivalent:
\begin{enumerate} 
\item There is an almost split sequence $0\to X\to Y\to Z\to 0$ in $\C$.
\item There is a simple object $S\in\Eff(\C,\Ab)$ such that
  $S(X)\neq 0$.
\end{enumerate}
\end{lem}
\begin{proof}
(1) $\Rightarrow$ (2): Use Lemma~\ref{le:almost-split}.

(2) $\Rightarrow$ (1): The functor $S$ admits a minimal projective presentation
\[0\lto\Hom_\C(Z,-)\lto\Hom_\C(Y,-)\lto \Hom_\C(X,-)\lto S\lto 0\] in $\Fp(\C,\Ab)$ 
since $\C$ is Krull-Schmidt. It follows from Proposition~4.4 in
Chap.~II of \cite{Au1978} that the underlying sequence  $0\to X\to
Y\to Z\to 0$ is almost split in $\C$.
\end{proof}

\subsection*{Length and support}

Let $\C$ be an essentially small additive category and suppose that
$\C$ is Krull-Schmidt. Let $\ind\C$ denote a representative set of the
isoclasses of indecomposable objects. For an additive functor
$F\colon\C\to\Ab$ set
\[\supp (F)=\{X\in\ind\C\mid FX\neq 0\}\]
and let $\ell(F)$ denote the composition length of $F$ in the category
of additive functors $\C\to\Ab$. 

\begin{lem}\label{le:fun-length}
  For an additive functor $F\colon\C\to\Ab$ we
  have \[\ell(F)=\sum_{X\in\ind\C}\ell_{\End_\C(X)}(FX).\]
\end{lem}
\begin{proof}
  Let $F\colon \C\to\Ab$ be a simple functor and $FX\neq 0$ for some
  $X\in\ind\C$. Then we have $\supp(F)=\{X\}$ and
  $\ell(F)=1=\ell_{\End_\C(X)}(FX)$. From this the assertion follows
  by induction on $\ell(F)$.
\end{proof}

\begin{lem}\label{le:supp-homfinite}
  Let $\C$ be Hom-finite and $F\colon\C\to\Ab$ a finitely generated
  functor. Then $\ell(F)$ is finite if and only if $\supp(F)$ is
  finite.
\end{lem}
\begin{proof}
Apply Lemma~\ref{le:fun-length}. Clearly, $\supp(F)$ is finite when
$\ell(F)$ is finite. For the converse observe that
$\ell_{\End_\C(X)}(FX)$ is finite for all $X\in\C$ since $\C$ is
Hom-finite and $F$ is the quotient of a representable functor.
\end{proof}

\begin{rem}
Let $\Fp(\C,\Ab)$ be abelian and $F\in \Fp(\C,\Ab)$. Then $\ell(F)$
does not depend on the ambient category since every simple object in
$\Fp(\C,\Ab)$ is simple in the category of all additive functors
$\C\to\Ab$.
\end{rem}

Let $\C$ be Hom-finite and fix an object $X\in\C$. The assignment
\[\chi_X\colon\Fp(\C,\Ab)\lto\bbZ, \quad F\mapsto \ell_{\End_\C(X)}(FX)\]
induces a homomorphism $K_0(\Fp(\C,\Ab))\lto\bbZ$.

\begin{lem}\label{le:supp-K0}
Let $\C$ be Hom-finite. Given functors  $F$ and $(F_i)_{i\in I}$  in $\Fp(\C,\Ab)$, 
\[ [F]\in\langle [F_i]\mid i\in I\rangle\subseteq K_0(\Fp(\C,\Ab))
\qquad\text{implies}\qquad
\supp(F)\subseteq\bigcup_{i\in I}\supp(F_i).\]
\end{lem}
\begin{proof}
  Fix $X\in\ind\C$. We have $X\not\in\supp(F)$ if and only if
  $\chi_X(F)=0$.  Thus $X\not\in\bigcup_{i\in I}\supp(F_i)$ implies
  $\chi_X(F_i)=0$ for all $i\in I$.  If $[F]$ is generated by the
  $[F_i]$, then this implies $\chi_X(F)=0$. Thus $X\not\in\supp(F)$.
\end{proof}

\subsection*{Relations for Grothendieck groups}

Let $\C$ be an essentially small abelian category and 
consider the Yoneda functor
\[\C\lto\Fp(\C,\Ab),\quad X\mapsto h_X:=\Hom_\C(X,-).\]
 
\begin{lem}\label{le:basis-K0}
  The Yoneda functor induces an isomorphism of abelian groups
\[K_0(\C,0)\xto{\sim} K_0(\Fp(\C,\Ab)).\]
\end{lem}
\begin{proof}
The Yoneda functors identifies $\C$ with the full subcategory of
projective objects in $\Fp(\C,\Ab)$. From this the assertion follows
since every object in $\Fp(\C,\Ab)$ admits a finite projective resolution.
\end{proof}

Given an almost split sequence $0\to X\to Y\to Z\to 0$ in $\C$, let
$S_X$ denote the corresponding simple functor in $\Fp(\C,\Ab)$ with
$\supp(S_X)=\{X\}$; see Lemma~\ref{le:almost-split2}.

\begin{lem}\label{le:ass-generate}
  Let $\C$ be an abelian category. Then the following are equivalent:
\begin{enumerate}
\item The kernel of $\pi\colon K_0(\C,0)\to K_0(\C)$
  is generated by elements $[X]-[Y]+[Z]$ that are given by almost
  split sequences $0\to X\to Y\to Z\to 0$ in $\C$.
\item $ [F]\in \langle [S_X]\mid 0\to X\to Y\to Z\to 0 \text{
    almost split}\rangle$ for all  $F\in \Eff(\C,\Ab)$.
\end{enumerate}
\end{lem}
\begin{proof}
An exact sequence  $\eta\colon 0\to X\to Y\to Z\to 0$ in $\C$ gives rise to
an exact sequence 
\[0\lto\Hom_\C(Z,-)\lto\Hom_\C(Y,-)\lto\Hom_\C(X,-)\lto F_\eta \lto 0\]
in  $\Fp(\C,\Ab)$ with $[F_\eta]=[h_X]-[h_Y]+[h_Z]$. The assertion
then follows by identifying $[X]$ with $[h_X]$ for all $X\in\C$,
keeping in mind  Lemma~\ref{le:basis-K0}.
\end{proof}

\begin{prop}\label{pr:eff-ass}
  Let $\C$ be an essentially small abelian Krull-Schmidt
  category. Consider the following conditions:
\begin{enumerate}
\item Every effaceable finitely presented functor $\C\to\Ab$ has finite
  length.
\item The kernel of $\pi\colon K_0(\C,0)\to K_0(\C)$
  is generated by elements $[X]-[Y]+[Z]$ that are given by almost
  split sequences $0\to X\to Y\to Z\to 0$ in $\C$.
\end{enumerate}
Then \emph{(1)} implies \emph{(2)} and the converse holds when $\C$ is Hom-finite.
\end{prop}
\begin{proof}
  (1) $\Rightarrow$ (2): Let $S$ be a simple composition factor of an
  effaceable functor. Choosing a minimal projective presentation of
  $S$ in $\Fp(\C,\Ab)$ gives rise to an almost split sequence
  $0\to X\to Y\to Z\to 0$ in $\C$ so that $S=S_X$; see
  Lemma~\ref{le:almost-split2}. Thus condition (2) in
  Lemma~\ref{le:ass-generate} holds.

(2) $\Rightarrow$ (1):  Fix $F\in\Eff(\C,\Ab)$. Then
  Lemmas~\ref{le:supp-K0} and \ref{le:ass-generate} imply that
  $\supp(F)$ is finite. From Lemma~\ref{le:supp-homfinite} it follows
  that $F$ has finite length.
\end{proof}
 
\begin{rem}
The property that every effaceable functor $\C\to\Ab$ has finite
length is self-dual, thanks to the duality \eqref{eq:eff}.
\end{rem}

\section{Effaceable functors and pure-injectives}

The aim of this section is a characterisation of the length categories
$\C$ such that every effaceable finitely presented functor $\C\to\Ab$
has finite length. This involves the study of pure-injective objects,
and we need to embed $\C$ into a Grothendieck category.

\subsection*{Locally finitely presented categories}

Let $\A$ be a Grothendieck category. An object $X\in\A$ is
\emph{finitely presented} if $\Hom_\A(X,-)$ preserves filtered
colimits, and we denote by $\fp\A$ the full subcategory of finitely
presented objects in $\A$. The category $\A$ is called \emph{locally
  finitely presented} if $\A$ has a generating set of finitely
presented objects \cite{Br1970}.

Let $\C$ be an essentially small abelian category. We denote by
$\Lex(\C^\op,\Ab)$ the category of left exact functors $\C^\op\to\Ab$
and set $\A=\Lex(\C^\op,\Ab)$. Observe that $\A$ is a locally finitely
presented Grothendieck category. The category $\C$ identifies
with $\fp\A$ via
the functor
\[\C\lto\A, \quad X\mapsto\Hom_\C(-,X),\]
and every object in $\A$ is a filtered colimit of objects in
$\C$. 

\subsection*{Locally noetherian categories}
 
A Grothendieck category $\A$ is called \emph{locally noetherian} if
$\A$ has a generating set of noetherian objects. In that case finitely
presented and noetherian objects in $\A$ coincide.

A Grothendieck category $\A$ is called \emph{locally finite} if $\A$
has a generating set of finite length objects.  When $\A$ is locally
finite, then every noetherian object has finite length, since any
object is the directed union of finite length subobjects. Thus
finitely presented and finite length objects in $\A$ coincide

\subsection*{Purity}

Let $\C$ be an essentially small abelian category and
$\A=\Lex(\C^\op,\Ab)$.  We recall the following construction from
\cite[\S3]{CB1994}.  Set $\check\C=\Fp(\C,\Ab)^\op$ and
$\bar\A=\Lex(\check\C^\op,\Ab)$. Observe that $\check\C$ is abelian
and identifies with $\fp\bar\A$. The functor
\[h\colon\C\lto\check\C, \quad X\mapsto\Hom_\C(X,-)\]
is right exact and extends to a colimit preserving and fully faithful functor
\[h_!\colon\A\lto\bar\A, \quad X\mapsto\bar X\] 
that makes the following square commutative:
\[\begin{tikzcd}
\C\arrow{r}{h}\arrow{d}&\check\C \arrow{d}\\
\A\arrow{r}{h_!}&\bar\A
\end{tikzcd}
\]
Note that $h_!$ is the left adjoint of $h^*\colon\bar\A\to\A$ given by
$h^*(X)=X\comp h$. 

There is a notion of purity for $\A$. A sequence
$0\to X\to Y\to Z\to 0$ in $\A$ is \emph{pure-exact} if $\Hom_\A(C,-)$
takes it to an exact sequence of abelian groups for all finitely
presented $C\in\A$. An object $M\in\A$ is \emph{pure-injective} if
every pure monomorphism $X\to Y$ induces a surjective map
$\Hom_\A(Y,M)\to\Hom_\A(X,M)$.

\begin{lem} 
\begin{enumerate}
\item A sequence
$0\to X\to Y\to Z\to 0$ in $\A$ is pure-exact if and only if the
induced sequence
$0\to \bar X\to \bar Y\to \bar Z\to 0$ in $\bar\A$ is exact.
\item The functor $X\mapsto \bar X$ identifies  the pure-injective objects
  in $\A$ with the injective objects in $\bar\A$.
\end{enumerate}
\end{lem}
\begin{proof}
See Lemma~4 in \S 3.3  and Lemma~1 in \S 3.5  of \cite{CB1994}. 
\end{proof}

The category $\bar\A$ has enough injective objects. Thus every object
in $\A$ admits a pure monomorphism into a pure-injective object. We
call such a morphism a \emph{pure-injective envelope} if it becomes an
injective envelope in $\bar\A$.

\begin{exm}\label{ex:hom-finite}
Suppose that $\C$ is Hom-finite. Then  every finitely
  presented object in $\A$ is pure-injective. This follows from Theorem~1 in \S 3.5
  of \cite{CB1994}.
\end{exm}

Let us write $\Ind\A$ for a representative set of the indecomposable
pure-injective objects in $\A$, containing exactly one object for each
isomorphism class. For a class $\X\subseteq\bar\A$ set
\[\X^\perp=\{M\in\Ind\A\mid\Hom_{\bar\A}(X,\bar M)=0\text{ for all
}X\in\X\}.\]

We recall the following detection result; see
\cite[Thm.~3.8]{He1997} and \cite[Thm.~4.2]{Kr1997}.

\begin{prop}\label{pr:perp}
  \pushQED{\qed} Let $\X,\Y$ be Serre subcategories of
  $\check\C$. Then
\[\X\subseteq\Y\quad\iff\quad \X^\perp\supseteq\Y^\perp.\qedhere\]
\end{prop}

\subsection*{Effaceable functors}

We compute the orthogonal complement of the category of effaceable functors.

\begin{lem}\label{le:eff-perp}
We have $\Eff(\C,\Ab)^\perp=\{M\in\Ind\A\mid M \text{ is injective}\}$.
\end{lem}
\begin{proof}
  Fix $F\in\Eff(\C,\Ab)$ with presentation \eqref{eq:fp} and
  $M\in\Ind\A$. Then we have $\Hom_{\bar\A}(F,\bar M)=0$ if and only if every
  morphism $X\to M$ factors through $X\to Y$. It follows that
  $\Hom_{\bar\A}(F,\bar M)=0$ when $M$ is injective. If $M$ is not
  injective, then there is a monomorphism $\a\colon M\to N$ in $\A$
  that does not split. Moreover, $\a$ is not a pure monomorphism since
  $M$ is pure-injective. Thus we may assume that $C=\Coker \a$ is in
  $\C$. Write $N=\colim_i N_i$ as a filtered colimit of objects in
  $\C$. Then for some $i$ the induced morphism $\b_i\colon N_i\to C$
  is an epimorphism. Let $\a_i\colon M_i\to N_i$ be the kernel of
  $\b_i$ and set $F_i=\Coker\Hom_\C(\a_i,-)$. Then $F_i$ is in
  $\Eff(\C,\Ab)$ and $\Hom_{\bar\A}(F_i,\bar M)\neq 0$ by
  construction.
\end{proof}

Let $\Fp(\C,\Ab)_0$ denote the full subcategory of finite length
objects in $\Fp(\C,\Ab)$.

\begin{lem}\label{le:fl-perp}
  Let $\C$ be a Krull-Schmidt category.  An object in $\Ind\A$ belongs
  to $(\Fp(\C,\Ab)_0)^\perp$ if and only if it is not the
  pure-injective envelope of the source of a left almost split
  morphism in $\C$.
\end{lem}
\begin{proof}
  An object $M\in\Ind\A$ belongs to $(\Fp(\C,\Ab)_0)^\perp$ if
  and only if $\Hom_{\bar\A}(S,\bar M)=0$ for every simple
  objects $S$ in $\check\C$. By Lemma~\ref{le:almost-split}, any
  simple object $S$ in $\check\C$ arises as the kernel of a
  morphism $\bar X\to \bar Y$ that corresponds to a left almost split
  morphism $X\to Y$ in $\C$. Moreover, the morphism $S\to \bar X$ is
  an injective envelope in $\check\C$ since $\End_\C(X)$ is local.
  It remains to observe that a morphism $X\to M$ in $\A$ is a
  pure-injective envelope if and only if $\bar X\to \bar M$ is an
  injective envelope in $\bar\A$.
\end{proof}

\begin{prop}\label{pr:eff}
  Let $\A$ be a locally finitely presented Grothendieck category and
  set $\C=\fp\A$. Suppose that $\C$ is an abelian Krull-Schmidt
  category. Then the following conditions are equivalent:
\begin{enumerate}
\item Every effaceable finitely presented functor $\C\to\Ab$ has
  finite length.
\item Every indecomposable pure-injective object in $\A$ is injective
  or the pure-injective envelope of the source of a left almost split
  morphism in $\C$.
\end{enumerate}
\end{prop}
\begin{proof}
  Effaceable functors and finite length functors form Serre
  subcategories in $\Fp(\C,\Ab)$. Their orthogonal complements in
  $\Ind\A$ are described in Lemmas~\ref{le:eff-perp} and
  \ref{le:fl-perp}. It remains to apply Proposition~\ref{pr:perp}.
\end{proof}

\subsection*{Fp-injective objects}

Let $\A$ be a locally finitely presented Grothendieck category. An
object $X\in\A$ is called \emph{fp-injective} if $\Ext^1_\A(C,X)=0$
for every finitely presented object $C\in\A$.

Let $\C$ be an essentially small abelian category and set
$\A=\Lex(\C^\op,\Ab)$.

\begin{lem}\label{le:fun-exact}
 A functor $X\in\A$ is exact if and only if $X$ is an fp-injective object.
\end{lem}
\begin{proof}
Using the identification $\C\xto{\sim}\fp\A$,
the functor $X$ is exact iff  for every exact sequence $\eta\colon 0\to A\to B\to
C\to 0$  in $\fp\A$ the induced sequence 
\[\Hom_\A(\eta,X) \colon 0\to\Hom_\A(C,X)\to\Hom_\A(B,X)\to\Hom_\A(A,X)\to 0\] is
exact. 

Now suppose   $\Ext^1_\A(C,X)=0$. Clearly, this implies exactness of
$\Hom_\A(\eta,X)$ for any exact  $\eta\colon 0\to A\to B\to
C\to 0$  in $\fp\A$. Conversely, let $\mu\colon 0\to X\to Y\to C\to 0$ be
exact in $\A$ and write $Y=\colim_i Y_i$ as filtered colimit of finitely
presented objects. This yields an exact sequence  $\mu_j\colon 0\to X_j\to Y_j\to
C\to 0$ in $\fp\A$ for some $j$. Now exactness of  $\Hom_\A(\mu_j,X)$
imlies that $\mu$ splits.
\end{proof}

\begin{lem}\label{le:fpinj}
Let $X\in\A$. Then $\bar X$ is fp-injective in $\bar\A=\Lex(\check\C^\op,\Ab)$.
\end{lem}
\begin{proof}
  We apply Lemma~\ref{le:fun-exact}.  Recall that
  $\check\C=\Fp(\C,\Ab)^\op$. Thus $\bar X$ is exact for
  $X=\Hom_\C(-,C)$ with $C\in\C$. Any object $X\in\A$ is a filtered colimit
  of representable functors. Thus it remains to observe that a
  filtered colimit of exact functors is exact.
\end{proof}

\begin{lem}\label{le:noeth}
  Let $\A$ be a locally noetherian Grothendieck category. Then every
  fp-injective object in $\A$ is injective and decomposes into
  indecomposable objects.
\end{lem}
\begin{proof} 
  When $\A$ is locally noetherian, then finitely presented and
  noetherian objects in $\A$ coincide and are therefore closed under
  quotients. Now apply Baer's criterion to show that fp-injective
  implies injective. The decomposition into indecomposables follows
  from an application of Zorn's lemma, using that fp-injective objects
  are closed under filtered colimits.
\end{proof}

\subsection*{Finite type}

We are now ready to extend some known characterisations of finite
representation type for module categories to more general abelian
categories, including the length categories of infinite height.

Recall that every essentially small abelian category $\C$ with all
objects in $\C$ noetherian embeds into a locally noetherian
Grothendieck category $\A$ with $\C\xto{\sim}\fp\A$.

\begin{thm}\label{th:indec}
  Let $\A$ be a locally noetherian Grothendieck category and set
  $\C=\fp\A$.  Suppose that $\C$ is Hom-finite. Then the following are
  equivalent:
\begin{enumerate}
\item Every effaceable finitely presented functor $\C\to\Ab$ has
  finite length.
\item The kernel of $\pi\colon K_0(\C,0)\to K_0(\C)$
  is generated by elements $[X]-[Y]+[Z]$ that are given by almost
  split sequences $0\to X\to Y\to Z\to 0$ in $\C$.
\item The category $\C$ has almost split sequences, and every non-zero
  object in $\A$ has an indecomposable direct summand that is finitely
  presented or injective.
\item The category $\C$ has almost split sequences, and every
  indecomposable object in $\A$  is finitely presented or
  injective.
\end{enumerate}
\end{thm}
\begin{proof} 
We identify $\check\C\xto{\sim}\fp\bar\A$.

(1) $\Leftrightarrow$ (2): See Proposition~\ref{pr:eff-ass}.

(1) $\Rightarrow$ (3): Let $X\neq 0$ be an object in $\A$. Suppose
first that $\Hom_{\bar\A}(S,\bar X)\neq 0$ for a simple and effaceable
$S\in\check\C$. Choose an injective envelope $\a\colon S\to \bar C$ in
$\check\C$. Then any non-zero morphism $\p\colon S\to\bar X$ factors
through $\a$, since $\bar X$ is fp-injective by
Lemma~\ref{le:fpinj}. On the other hand, $\a$ factors through $\p$
since $C$ is pure-injective by Example~\ref{ex:hom-finite}. Thus $C$
is isomorphic to a direct summand of $X$. Now suppose that
$\Hom_{\bar\A}(F,\bar X) =0$ for all effaceable $F\in\check\C$. Then
$\bar X\colon\Fp(\C,\Ab)\to\Ab$ vanishes on $\Eff(\C,\Ab)$ and
identifies with an exact functor $\C^\op\to\Ab$ via the functor
\eqref{eq:eff-formula}. Thus $X$ is injective and has an
indecomposable summand since $\A$ is locally noetherian; see
Lemma~\ref{le:noeth}.

Let $X\in\C$ be an indecomposable non-injective object. Then there is
an epimorphism $\Hom_\C(X,-)\to F$ with $F\neq 0$ effaceable. The
object $F$ has finite length and we may assume that $F$ is simple.  It
follows from Lemma~\ref{le:almost-split2} that there is an almost
split sequence $0\to X\to Y\to Z\to 0$ in $\C$.

For every indecomposable non-projective $Z\in\C$, there is an almost
split sequence $0\to X\to Y\to Z\to 0$ in $\A$, by Theorem~1.1 in
\cite{Kr2016}. The object $X$ is indecomposable, and therefore in
$\C$, by the first part of the proof.

 (3) $\Rightarrow$ (4): Clear. 

 (4) $\Rightarrow$ (1):  Use Proposition~\ref{pr:eff}.
\end{proof}

\begin{rem}\label{re:injectives}
  Suppose the conditions in Theorem~\ref{th:indec} hold. Then effaceable
  and finite length functors agree if and only if every injective
  object in $\C$ is zero. This follows from the proof of Proposition~\ref{pr:eff}.
\end{rem}

\begin{rem}
Theorem~\ref{th:indec} generalises various known characterisations of
finite representation type for module categories. 

For a ring $\La$, let $\mod\La$ denote the category of finitely
presented  $\La$-modules. Recall that $\La$ has \emph{finite
  representation type} if $\mod\La$ is a length category with only
finitely many isomorphism classes of indecomposable objects.

(1) A ring $\La$ has finite representation type if and only if every
finitely presented functor $\mod\La\to\Ab$ has finite length  \cite{Au1974}.

(2) An artin algebra $\La$ has finite representation type if and only
if for $\C=\mod\La$ the kernel of $\pi\colon K_0(\C,0)\to K_0(\C)$ is
generated by elements $[X]-[Y]+[Z]$ that are given by almost split
sequences $0\to X\to Y\to Z\to 0$ in $\C$
\cite{Au1984,Bu1981}.\footnote{The proof of Theorem~\ref{th:indec} is
  close to Butler's original proof. Auslander's proof is based on the
  use of a bilinear form on $K_0(\C,0)$, following work of Benson and
  Parker on Green rings \cite{BP1984}.}

  (3) An artin algebra $\La$ has finite representation type if and
  only if every indecomposable $\La$-module is finitely presented
  \cite{Au1974,Au1976,RT1974}. 
\end{rem}

\begin{exm}
Fix a field  and consider the category of nilpotent finite
dimensional representations of the following quiver with relations:
\[
\begin{tikzcd}[cells={nodes={}}]
  \arrow[loop left, distance=2em, start anchor={[yshift=-1ex]west},
  end anchor={[yshift=1ex]west}]{}{\a} 
  \scriptstyle{\circ} &
  \scriptstyle{\circ}\arrow[swap]{l}{\b}&(\a\b=0)
\end{tikzcd}
\]
We denote this length category by $\C$ and observe that it is
\emph{discrete}: the fibres of the map $K_0(\C,0)\to K_0(\C)$ are finite.
However, the finiteness
conditions in Theorem~\ref{th:indec} are not satisfied. For instance,
the unique injective and simple object is not the end term of an
almost split sequence in $\C$.
\end{exm}

\section{Uniserial categories}

In this section we establish for uniserial categories the finiteness
conditions in Theorem~\ref{th:indec}. Moreover, we show that uniserial
categories of infinite height are precisely the length categories such
that effaceable and finite length functors agree.

Let $\C$ be an essentially small abelian category. Set
$\A=\Lex(\C^\op,\Ab)$ and identify $\C$ with the full subcategory of
finitely presented objects in $\A$.

\subsection*{Length and height}

Fix an object $X\in\A$.  The composition length of $X$ is denoted by
$\ell(X)$. The \emph{socle series} of $X$ is the chain of subobjects
\[0=\soc^0(X) \subseteq\soc^1(X) \subseteq\soc^2(X) \subseteq\ldots\]
of $X$ such that $\soc^1(X)$ is the \emph{socle} of $X$ (the largest
semisimple subobject of $X$) and $\soc^{n+1}(X)$ is given by
$\soc^{n+1}(X)/\soc^{n}(X)=\soc^1(X/\soc^n(X))$.  We set
$\height (X)= n$ when $n$ is the smallest integer such that
$\soc^n(X)=X$, and $\height(X)=\infty$ when such $n$ does not exist.

Let $X=\colim X_i$ be written as a filtered colimit in $\A$. Then
$\soc^n(X)=\colim\soc^n(X_i)$ for all $n\ge 0$. Thus
$X=\bigcup_{n\ge 0}\soc^n(X)$ when every object in $\C$ has finite
length.

\subsection*{Uniserial categories}

Recall that $\C$ is \emph{uniserial} if $\C$ is a length category and
each indecomposable object has a unique composition series.

\begin{lem}\label{le:uniserial-height}
Let $\C$ be an abelian length category. Then $\C$ uniserial if and
only if 
$\height(X)=\ell(X)$  for every indecomposable $X\in\C$.
\end{lem}
\begin{proof}
Let  $X\in\C$ be indecomposable. If $\height(X)=\ell(X)$, then the
socle series of $X$ is the unique composition series of $X$.

Now assume $\height(X)\neq\ell(X)$. Then there exists some $n\ge 0$
such that   \[\soc^{n+1}(X)/\soc^{n}(X)=S_1\oplus \ldots\oplus S_r\] with all
  $S_i$  simple and $r>1$. Choose $n$ minimal and let
  $\soc^{n}(X) \subseteq U_i\subseteq X$ be given by
  $U_i/\soc^{n}(X)=S_i$ Then we have at least $r$ different composition series
  \[0=\soc^0(X) \subseteq\soc^1(X) \subseteq\ldots\subseteq
  \soc^{n}(X) \subseteq U_i\subseteq\ldots\subseteq \soc^{n+1}(X)\subseteq\ldots\]
  of $X$. 
\end{proof}

\begin{lem}\label{le:uniserial-Hom-finite}
Let $\C$ be a uniserial category. Then $\C$ is Hom-finite.
\end{lem}
\begin{proof}
  We need to show for all $X,Y$ in $\C$ that the $\End_\C(Y)$-module
  $\Hom_\C(X,Y)$ has finite length. It suffices to assume that $Y$ is
  indecomposable; see Lemma~\ref{le:Hom-finite}.  We claim that
  \[\ell_{\End_\C(Y)}(\Hom_\C(X,Y))\le \ell(X).\] 
 Using induction on $\ell(X)$ the claim  reduces to the
  case that $X$ is simple. So let $S=\soc(Y)$ and write $E=E(Y)$ for
  its injective envelope. Note that $\soc^n(E)=Y$ for $n=\ell(Y)$, by
  Lemma~\ref{le:uniserial-height}. Thus any endomorphism $E\to E$
  restricts to a morphism $Y\to Y$.  Write $i\colon S\to Y$ for the
  inclusion. Then any morphism $j\colon S\to Y$ induces an
  endomorphism $f\colon E\to E$ such that $f|_Y\comp i=j$. Thus the
  $\End_\C(Y)$-module $\Hom_\C(S,Y)$ is cyclic, and it is annihilated
  by the radical of $\End_\C(Y)$. Therefore $\Hom_\C(S,Y)$ is simple.
\end{proof}

\begin{prop}\label{pr:uniserial-eff}
  Let $\C$ be a uniserial category. Then every non-zero object in $\A$ has an
  indecomposable direct summand that belongs to $\C$ or is injective.
\end{prop}
\begin{proof}
  From Lemma~\ref{le:uniserial-height} it follows that for every
  indecomposable injective object $E$ in $\A$ the subobjects form a
  linear chain
\[0=E_0\subseteq E_1\subseteq E_2\subseteq \ldots\]
with $E_n=\soc^n(E)$ in $\C$ for all $n\ge 0$ and $E=\bigcup_{n\ge
  0}E_n$. Note that  $E=E_{\ell(E)}$ when $\ell(E)<\infty$.

  Fix $X\neq 0$ in $\A$ and choose a simple subobject $S\subseteq X$.
  Let $U\subseteq X$ be a maximal subobject containing $S$ such that
  $S\subseteq U$ is essential; this exists by Zorn's lemma. Then $U$
  is injective or belongs to $\C$. In the first case we are done. So
  assume $U\in\C$. We claim that the inclusion $U\to X$ is a pure
  monomorphism. To see this, choose a morphism $C\to X/U$ with
  $C\in\C$. This yields the
  following commutative diagram with exact rows.
\[\begin{tikzcd}
0\arrow{r}&U\arrow{r}\arrow[equals]{d}&V\arrow{r}\arrow{d}&C\arrow{r}\arrow{d}&0\\
0\arrow{r}&U\arrow{r}&X\arrow{r}&X/U\arrow{r}&0
\end{tikzcd}\]
Write $V=\bigoplus_i V_i$ as a direct sum of indecomposable
objects. Then there exists an index $i$ such that the composite
$S\hookrightarrow U\to V_i\to X$ is non-zero. Thus $S\to V_i$ is
essential and $V_i\to X$ is a monomorphism. It follows from the
maximality of $U$ that $U\to V_i$ is an isomorphism. Therefore
the top row splits, and this yields the claim. It remains to observe
that every object in $\C$ is pure-injective since $\C$ is Hom-finite;
see Example~\ref{ex:hom-finite} and
Lemma~\ref{le:uniserial-Hom-finite}
\end{proof}

Let $\Fp(\C,\Ab)_0$ denote the full subcategory of finite length
objects in $\Fp(\C,\Ab)$.

\begin{cor}\label{co:uniserial}
  Let $\C$ be a uniserial category. Then every effaceable finitely
  presented functor $\C\to\Ab$ has finite length. Moreover effaceable
  and finite length functors agree if and only if all injective
  objects in $\C$ are zero. In that case  we have an equivalence
\begin{equation}\label{eq:eff-uni}
\frac{\Fp(\C^\op,\Ab)}{\Fp(\C^\op,\Ab)_0}\stackrel{\sim}\lto \C.
\end{equation}
\end{cor}
\begin{proof}
  The first assertion follows from Proposition~\ref{pr:uniserial-eff}
  and Theorem~\ref{th:indec}, keeping in mind that $\C$ is Hom-finite
  by Lemma~\ref{le:uniserial-Hom-finite}. Effaceable and finite length
  functors agree if and only if all injective objects in $\C$ are
  zero, by Remark~\ref{re:injectives}. Having this property, the
  equivalence is \eqref{eq:eff-formula}.
\end{proof}

\begin{rem}
  An interesting instance of the equivalence \eqref{eq:eff-uni} arises
  from the study of Greenberg modules; see \cite[V, \S4, 1.8]{DG1970}
  and \cite[\S\S4-5]{Sc1970}.
\end{rem}

\begin{rem}
 It would be interesting to have a more direct proof of
 Corollary~\ref{co:uniserial}, avoiding the embedding of $\C$ into a
 Grothendieck category.
\end{rem}

\subsection*{Serre duality}

Our next aim is a characterisation of uniserial categories that
involves Serre duality. To this end recall the following
characterisation in terms of Ext-quivers \cite{AR1968}.

\begin{prop}\label{pr:Gabriel}
  A length category $\C$ is uniserial if and only it satisfies the
  following condition and its dual: For each simple object $S$ there
  exists, up to isomorphism, at most one simple object $T$ such that
  $\Ext_\C^1(S,T)\neq 0$, and in this case
  $\ell_{\End_\C(T)}(\Ext_\C^1(S,T))=1$.\qed
\end{prop}

We fix a field $k$ and write $D=\Hom_k(-,k)$ for the usual duality.

Let $\C$ be a $k$-linear abelian category such that $\Hom_\C(X,Y)$ is
finite dimensional for all $X,Y\in\C$.  Following \cite{BK1989}, the
category $\C$ satisfies \emph{Serre duality} if there exists an
equivalence $\tau\colon \C \xto{\sim} \C$ with a functorial $k$-linear
isomorphism
\[D\Ext^1_\C(X,Y)\xto{\sim} \Hom_\C(Y, \tau X)\quad\text{for all}\quad
  X,Y\in\C.\] The functor $\tau$ is called \emph{Serre functor} or
\emph{Auslander-Reiten translation}. Note that a Serre functor is
$k$-linear and essentially unique provided it exists; this follows
from Yoneda's lemma.

The following result is well-known and describes the structure of a
length category with Serre duality. Let us recall the shape of the
relevant diagrams.
\begin{align*} 
\tilde A_{n}&\colon\;\;\; 
\begin{tikzcd}[ampersand replacement=\&]
1\arrow[dash]{r}\& 2\arrow[dash]{r}\&3\arrow[dash]{r} \&\cdots
\arrow[dash]{r}\& n\arrow[dash]{r}\&n+1 \arrow[dash, bend
left=12]{lllll}
\end{tikzcd}
\\[1ex]  A_\infty^\infty&\colon\;\;\; 
\begin{tikzcd}[ampersand replacement=\&]
\cdots \arrow[dash]{r}\& \scriptstyle{\circ}\arrow[dash]{r}\&
 \scriptstyle{\circ}\arrow[dash]{r}\&
 \scriptstyle{\circ}\arrow[dash]{r}\&
 \scriptstyle{\circ}\arrow[dash]{r}\&\cdots
\end{tikzcd}
\end{align*}

\begin{prop}\label{pr:serre}
Let $\C$ be a Hom-finite $k$-linear length category and suppose $\C$
admits a Serre functor $\tau$. Then $\C$ is uniserial. The category
$\C$ admits a unique decomposition $\C=\coprod_{i\in I} \C_i$ into
connected uniserial categories with Serre duality, where the index set
equals the set of $\tau$-orbits of simple objects in $\C$.  The
Ext-quiver of each $\C_i$ is either of type $ A_\infty^\infty$ (with
linear orientation) or of type $\tilde A_{n}$ (with cyclic
orientation).
\end{prop}

\begin{proof}
  We apply the criterion of Proposition~\ref{pr:Gabriel} to show that
  $\C$ is uniserial. To this end fix a simple object $S$. Then
  $\Ext_\C^1(S, T)\cong D\Hom_\C(T, \tau S) \neq 0$ for some
  simple object $T$ if and only if $T\cong\tau S$. Moreover,
  $\ell_{\End_\C(\tau S)}(\Ext_\C^1(S, \tau S))=1$. A quasi-inverse of
  $\t$ provides a Serre functor for $\C^\op$. Thus the category $\C$ is
  uniserial.

  The structure of the Ext-quiver of $\C$ follows from
  Proposition~\ref{pr:Gabriel}. The Serre functor acts on the set of
  vertices and the $\tau$-orbits provide the index set of the
  decomposition $\C=\coprod_{i\in I} \C_i$ into connected
  components. The Ext-quiver of $\C_i$ is of type $A_\infty^\infty$ if
  the corresponding $\tau$-orbit is infinite. Otherwise, the
  Ext-quiver of $\C_i$ is of type $\tilde A_n$ where $n+1$ equals the
  cardinality of the $\tau$-orbit.
\end{proof}

\subsection*{Auslander-Reiten duality}

For an abelian category $\A$ let $\St\A$ denote the \emph{stable
  category modulo injectives}, which is obtained from $\A$ by
identifying two morphisms $\p,\p'\colon X\to Y$ if
\[\Ext^1_\A(-,\p)=\Ext^1_\A(-,\p').\] 
When $\A$ has enough injective objects this means $\p-\p'$ factors
through an injective object. We write $\oHom_\A(-,-)$ for the
morphisms in $\St\A$.

Let us recall from \cite[Corollary~2.13]{Kr2016b} the following version of
Auslander-Reiten duality for Grothendieck categories, generalising the
usual duality for modules over artin algebras \cite{AR1975}.

\begin{prop}\label{pr:AR-duality}
  \pushQED{\qed} Let $\A$ be a $k$-linear and locally finitely
  presented Grothendieck category. There exist a functor
  $\t\colon\fp\A\to\St\A$ with a natural isomorphism
\begin{equation}\label{eq:ARformula}
D\Ext^1_\A(X,-)\cong\oHom_\A(-,\t X)\quad\text{for all}\quad
X\in\fp\A.\qedhere
\end{equation}
\end{prop}

\subsection*{Uniserial categories of infinite height}

We are now ready to characterise uniserial categories of infinite
height in terms of finitely presented functors.

\begin{thm}\label{th:uniserial}
  Let $\C$ be a $k$-linear length category such that $\Hom_\C(X,Y)$ is
  finite dimensional for all $X,Y\in\C$. Then the following are
  equivalent:
\begin{enumerate}
\item A finitely presented functor $\C\to\Ab$ is effaceable if and
  only if it has finite length.
\item The category $\C$ has Serre duality.
\item The category $\C$ is uniserial and all connected components have
  infinite height.
\end{enumerate}
\end{thm}

\begin{proof}
Set $\A=\Lex(\C^\op,\Ab)$ and identify $\C\xto{\sim}\fp\A$.

(1) $\Rightarrow$ (2):  We claim that the functor
$\t\colon\C\to\St\A$ in Proposition~\ref{pr:AR-duality} yields a
Serre functor for $\C$.

First observe that $\oHom_\A(-,X)=\Hom_\A(-,X)$ for every
$X\in\C$. Condition (1) implies that every injective object in $\C$ is
zero; see Remark~\ref{re:injectives}. Now let $\p\colon I\to X$ be a
morphism in $\A$ with indecomposable injective $I$. Then $\Ker\p$ is
indecomposable, and therefore injective or finitely presented, by
Theorem~\ref{th:indec}. Thus $\p=0$.

The assumption on $\C$ in (1) is also satisfied by $\C^\op$, thanks
to the duality \eqref{eq:eff}. Thus $\uHom_\C(X,-)=\Hom_\C(X,-)$ for every
$X\in\C$. 

The functor $\t\colon\C\to\St\A$ in Proposition~\ref{pr:AR-duality}
lands in $\C$, because all indecomposable objects in $\St\A$ belong to
$\C$ by Theorem~\ref{th:indec}. In fact, the formula
\eqref{eq:ARformula} yields an almost split
sequence $0\to\t Z \to Y\to Z\to 0$ in $\C$ for every indecomposable object $Z$. Thus
$\t\colon\C\to\C$ is essentially surjective on objects, since the
category $\C$ has almost split sequences by
Theorem~\ref{th:indec}. The defining isomorphism for $\t$ shows that
$\t$ is fully faithful, since
\[\uHom_\C(-,-)=\Hom_\C(-,-)=\oHom_\C(-,-).\] Indeed,
the induced map $\t_{X,Y}\colon \Hom_\C(X,Y)\to\Hom_\C(\t X,\t Y)$ is
a monomorphism. On the other hand, $\t$ induces maps
\[\uHom_\C(X,Y)\xto{\sim}D^2 \uHom_\C(X,Y)\to D\Ext^1_\C(Y,\t
  X)\xto{\sim}\oHom_\C(\t X,\t Y)\] where the middle one is the dual of
the monomorphism \[\Ext_\A^1(-,\t X)\lto D\uHom_\A(X,-)\] from
\cite[Theorem~2.15]{Kr2016b}. Thus $\t_{X,Y}$ is bijective.

(2) $\Rightarrow$ (3): Use Proposition~\ref{pr:serre}.

(3) $\Rightarrow$ (1):  Use Corollary~\ref{co:uniserial}.
\end{proof}

\section{Minimal length categories of infinite height}

Throughout this section we fix an algebraically closed field $k$. 

Our aim is an explicit description of the length categories of
infinite height that are \emph{minimal} in the sense that every proper
closed subcategory has only finitely many isoclasses of indecomposable
objects.  Here, a full subcategory of an abelian category is
\emph{closed} if it is closed under subobjects and quotients.

\begin{thm}\label{th:minimal}
Let $\C$ be a $k$-linear length category with the following properties:
\begin{enumerate}
\item The category $\C$ has only finitely many isoclasses of simple objects.
\item The spaces $\Hom_\C(X,Y)$ and $\Ext_\C^1(X,Y)$ are finite
  dimensional for all objects $X,Y$ in $\C$.
\item The category $\C$ has infinite height.
\item Every proper closed subcategory of $\C$ has only finitely
  many isoclasses of indecomposable objects.
\end{enumerate}
Then $\C$ is equivalent to the category of nilpotent finite
dimensional representations of some cycle
\[\begin{tikzcd}[ampersand replacement=\&]
Z_n\colon\&1\arrow{r}\& 2\arrow{r}\&3\arrow{r} \&\cdots
\arrow{r}\& n-1\arrow{r}\&n \arrow[bend
left=12]{lllll}\&(n\ge 1).
\end{tikzcd}\]
\end{thm}

We do not prove the theorem as it stands. Instead we switch to an
equivalent formulation involving representations of admissible
algebras.

\subsection*{Admissible algebras}

Let $A$ be a $k$-algebra and denote by $R$ its radical. We call $A$
\emph{admissible} if it satisfies the following two conditions:
\begin{enumerate}
\item[(a)] The algebras $A/R$ and $R/R^2$ are finite dimensional.
\item[(b)] The algebra $A$ is separated and complete for the $R$-adic
topology: $A\xto{\sim}\varprojlim_{l\geq0}\,A/R^l$, canonically.
\end{enumerate}

In the sequel, every admissible algebra $A$ will be considered
together with its $R$-adic topology. Accordingly, an ideal $I$ of $A$
is \emph{open} if $I$ contains some power $R^l$ iff $I$ is of finite
codimension. In other words, our admissible algebras are exactly the
profinite algebras satisfying (a). For more on profinite algebras, one
may consult \cite{Se1993}. The closure of $I$ equals the
intersection of all $I+R^l$.

Given an admissible algebra $A$, we are really interested in the
category $\mod_0 A$ of (left) $A$-modules of finite length. Condition
(a) ensures that $\mod_0 A$ has only finitely many isoclasses of
simples and that the spaces $\Hom_A(M,N)$ and $\Ext_A^1(M,N)$ are
finite dimensional for $M,N$ in $\mod_0 A$. Condition (b) comes up
naturally when one tries to recover $A$ from $\mod_0 A$, up to Morita
equivalence; see \cite{Ga1962}.

 \subsection*{Complete path algebras}

 Let $Q$ be a finite quiver. The \emph{complete path algebra}
 $k\llbracket Q\rrbracket$ consists of the formal series $\sum_u a_uu$
 where $u$ runs through the paths of $Q$ and $a_u\in k$. Here, the
 paths $\epsilon_i$ of length $0$, corresponding to the vertices $i$,
 are included. The multiplication is defined by
 \[\left(\sum_u a_uu\right)\left(\sum_v b_vv\right)=
   \sum_w \left(\sum_{uv=w}a_ub_v\right) w.\] For any integer
 $l\geq0$, the elements $\sum_u a_uu$, where $u$ is restricted to the
 paths of length $\geq l$, form an ideal
 $k\llbracket Q\rrbracket^{\geq l}$ of $k\llbracket Q\rrbracket$; this
 ideal is precisely the $l$-th power of the radical of
 $k\llbracket Q\rrbracket$. Consequently $k\llbracket Q\rrbracket$ is
 admissible and, according to our convention, will always be
 considered together with its $k\llbracket Q\rrbracket^{\geq1}$-adic
 topology.

\subsection*{Description of admissible algebras by
quiver and relations}

Any admissible $k$-algebra $A$ that is \emph{basic}, i.e., the
quotient $A/R$ is a finite direct product of copies of $k$, can be
presented as the complete path algebra of a finite quiver modulo a
closed ideal: choose a decomposition $1_A=e_1+\dots+e_n$ of the unit
element into pairwise orthogonal primitive idempotents; such a
decomposition can be obtained by lifting the unique one for $A/R$. The
quiver $Q_A$ then has vertices $1,\dots,n$, and there are arrows
$\alpha_{ji}^m\colon i\to j$, $1\leq m\leq n_{ji}$, where $n_{ji}$ is
the dimension of $e_j(R/R^2)e_i$.  Choose further elements
$a_{ji}^m\in e_jRe_i$, $1\leq m\leq n_{ji}$,  whose classes form a
basis of $e_j(R/R^2)e_i$. The choices uniquely determine a continuous
homomorphism $k\llbracket Q_A\rrbracket\to A$ that maps each
$\epsilon_i$ to $e_i$ and each $\alpha_{ji}^m$ to $a_{ji}^m$. This
homomorphism is surjective; its kernel $I$ is contained in
$k\llbracket Q_A\rrbracket^{\geq2}$ and necessarily closed. The
presentation $A\xleftarrow{\sim} k\llbracket Q_A\rrbracket/I$ allows
one to interprete a module in $\mod_0 A$ as a finite dimensional
representation of $Q_A$ that satisfies the relations of some
sufficiently small ideal $I+R^l$, depending on the module.

\subsection*{Representation types}

Let $A$ be an admissible $k$-algebra.

The algebra $A$ is \emph{representation-finite} if $\mod_0 A$ has
only finitely many isoclasses of indecomposables. In this case, $A$ is
necessarily finite dimensional. Otherwise, there is some
infinite dimensional indecomposable projective $P$, having a simple
top, and  $P$ admits indecomposable quotients of any finite length
$\geq1$.

The algebra $A$ is \emph{mild} if any quotient $A/I$ by some closed
ideal $I\neq0$ is representation-finite. 

\begin{rem}
  In the definition of `mild', the restriction to closed ideals is not
  necessary: if $I$ is an arbitrary ideal with closure $\bar I$, the
  algebras $A/\bar I$ and $A/I$ have the same finite dimensional
  modules. On the other hand, if $A/I$ is representation-finite, $I$
  is even open.
\end{rem}

\subsection*{The main result}

Very special but frequently encountered examples of admissible
algebras are the complete path algebras $k\llbracket Z_n\rrbracket$ of
the cyclic quivers
\[\begin{tikzcd}[ampersand replacement=\&]
Z_n\colon\&1\arrow{r}{\a_1}\& 2\arrow{r}{\a_2}\&3\arrow{r}{\a_3} \&\cdots
\arrow{r}{\a_{n-2}}\& n-1\arrow{r}{\a_{n-1}}\&n \arrow[bend
left=12]{lllll}{\a_n}\&(n\ge 1).
\end{tikzcd}\] They admit an alternative description: consider the
discrete valuation $k$-algebra $\frako=k\llbracket T\rrbracket$ with
maximal ideal $\frakm=Tk\llbracket T\rrbracket$ and fraction field
$K$.  Then $k\llbracket Z_n\rrbracket$ is isomorphic to the following
hereditary $\frako$-order in the simple $K$-algebra $M_n(K)$:
\[
\begin{bmatrix} \frako & \frakm & \frakm & \cdot & \frakm & \frakm \\
                          \frako & \frako & \frakm & \cdot & \frakm & \frakm \\
                          \frako & \frako & \frako & \cdot & \frakm & \frakm \\
                          \cdot & \cdot & \cdot & \cdot & \cdot & \cdot \\
                          \frako & \frako & \frako & \cdot & \frako & \frakm \\
                          \frako & \frako & \frako & \cdot & \frako & \frako 
\end{bmatrix}
\]

The Auslander-Reiten quiver of $\mod_0 k\llbracket Z_n\rrbracket$ is identified with
$\bbZ A_{\infty}/\tau^n$. If the orientation $\cdots\to  3\to 2\to1$ of
$A_{\infty}$ is used, then the vertex $(-i,l)$ corresponds to the uniserial
indecomposable of length $l$ with top located at $i$.

It is well-known that $k\llbracket Z_n\rrbracket$ is mild.

\begin{thm}
If a $k$-algebra is admissible, basic,
infinite dimensional, and mild, then it is isomorphic to 
$k\llbracket Z_n\rrbracket$ for some $n\ge 1$.
\end{thm}

Our result is close to well-known characterisations of hereditary
orders; see \cite{Gu1987}. The only novelty is that we can do without
assuming a `large' center or some `purity' condition \cite{Dr1990}.

\begin{proof}
Let the $k$-algebra $A$ be admissible, basic, infinite dimensional,
and mild. We denote by $R$ the radical of $A$ and fix a decomposition
$1_A=e_1+\cdots+e_n$ into pairwise orthogonal primitive idempotents.

Step 1: \emph{For any $i$, the algebra $e_iAe_i$ is a quotient of
  $k\llbracket T\rrbracket$, and for any $i,j$, the space $e_jAe_i$ is
  cyclic as a right module over $e_iAe_i$ or as a left module over
  $e_jAe_j$.} Indeed, by well-known observations due to Jans
\cite{Ja1957} and Kupisch \cite{Ku1965}, a similar statement is true
for each of the representation-finite algebras $A/R^l$. Our claim
follows by passage to the limit.

Step 2: The $e_jAe_j$-$e_iAe_i$-bimodule $e_jAe_i$ carries two natural
filtrations. The first one is its radical filtration given by the
subbimodules \[(e_jAe_i)^{\geq l} =\sum_{r+s=l}
(e_jRe_j)^sA(e_iRe_i)^r.\] By step~1 it coincides with the radical
filtration of $e_jAe_i$ over $e_iAe_i$ or over $e_jAe_j$, it contains
each non-zero subbimodule of $e_jAe_i$ exactly once, and each quotient
$(e_jAe_i)^{\geq l}/(e_jAe_i)^{\geq l+1}$ has dimension $\leq1$.

The second one is the filtration $(e_jR^me_i)_{m\geq0}$ induced by the
$R$-adic filtration on $A$. Each term, being a subbimodule of
$e_jAe_i$, appears in the first filtration, and each quotient
$e_jR^me_i/e_jR^{m+1}e_i$ is semisimple over $e_iAe_i$ and over
$e_jAe_j$.

Remembering that the second filtration is separated, we conclude that
it is obtained from the first one by possibly introducing repetitions
and that both filtrations define the same topology.

Step 3: Let $e$ be an idempotent of $A$ that is the sum of some of
$e_1,\dots,e_n$. We claim that $eAe$ {\it is admissible}: indeed, we
just saw in step~2 that each space $e_j(R/ReR)e_i$ (with $e_i$ and
$e_j$ occurring in $e$) has dimension at most $1$. Arguing as above, the
topologies on $e_jAe_i$ induced by the $eRe$-adic filtration on $eAe$
and the $R$-adic filtration on $A$ coincide. Therefore
$eAe\xto{\sim}\varprojlim_{l\geq0} eAe/(eRe)^l$, canonically.

Of course, our claim would be wrong without the assumption of mildness
of $A$.

Step 4: Let $e$ be as in step~3. We claim that $eAe$ \emph{is
  mild}. Indeed, let $J$ be a non-zero ideal of $eAe$ and $I=AJA$ its
extension to $A$. Since $A/I$ is in particular finite dimensional, the
fully faithful left adjoint $N\mapsto (A/I)e\otimes_{eAe/J}N$ of the
restriction functor $M\mapsto eM$ then maps $\mod_0 eAe/J$ into
$\mod_0 A/I$. Since $A/I$ is representation-finite, so is
$eAe/J$. Consequently, $eAe$ is mild.

Step 5: \emph{For at least one $i$, the algebra $e_iAe_i$ is
infinite dimensional, i.e., isomorphic to $k\llbracket T\rrbracket$.} Otherwise all
$e_jAe_i$ and therefore $A$ itself are finite dimensional by step~1.

Step 6: Consider the case $n=2$ (thus $1_A=e_1+e_2$) and assume for
definiteness that $e_1Ae_1$ is isomorphic to $k\llbracket T\rrbracket$ (step~5). Then
$e_1Ae_2Ae_1\neq0$: otherwise $A/Ae_2A$ is still isomorphic to
$k\llbracket T\rrbracket$ and representation-infinite, contradicting the mildness of
$A$. Therefore $e_1Ae_2Ae_1=(e_1Re_1)^r$ for some uniquely determined
integer $r\geq1$. We now transform the powers
\begin{equation}\label{eq:club}
(e_1Re_1)^l
\end{equation}
into ideals of $e_2Ae_2$:
\begin{equation}\label{eq:spade} 
e_2A(e_1Re_1)^lAe_2
\end{equation}
and then back into ideals of $e_1Ae_1$:
\begin{equation}\label{eq:heart} 
e_1Ae_2A(e_1Re_1)^lAe_2Ae_1=(e_1Re_1)^r(e_1Re_1)^l(e_1Re_1)^r=(e_1Re_1)^{l+2r}.
\end{equation}
Since the sequence of ideals in \eqref{eq:heart} is strictly
decreasing, so is the one in \eqref{eq:spade}. In particular, also
$e_2Ae_2$ is isomorphic to $k\llbracket T\rrbracket$ (step~5) and
$e_2Ae_1Ae_2\neq0$. Interchanging $e_1$ and $e_2$, we have
$e_2Ae_1Ae_2=(e_2Re_2)^s$ for some uniquely determined integer
$s\geq1$. We get
\[ (e_2Re_2)^{2s}=e_2Ae_1Ae_2Ae_1Ae_2=e_2A(e_1Re_1)^rAe_2
\]
and
\[ s=\dim (e_2Re_2)^s/(e_2Re_2)^{2s}=\dim
  e_2Ae_1Ae_2/e_2A(e_1Re_1)^rAe_2\geq r.
\]
By symmetry we even have $r=s$. If $r=s\geq2$, the quiver of $A$ is
\[
\begin{tikzcd}[cells={nodes={}}]
  \arrow[loop left, distance=2em, start anchor={[yshift=-1ex]west},
  end anchor={[yshift=1ex]west}]{} \arrow[yshift=0.7ex]{r}
  \scriptstyle{\circ} &
  \scriptstyle{\circ}\arrow[yshift=-0.7ex]{l}\arrow[loop right,
  distance=2em, start anchor={[yshift=1ex]east}, end
  anchor={[yshift=-1ex]east}]{}
\end{tikzcd}
\]
and already $A/R^2$ is representation-infinite: contradiction! Thus
$r=s=1$ and $A$ is isomorphic to $k\llbracket Z_2\rrbracket$.

Step 7: Since $A$ clearly is connected, one immediately deduces from
step~6: \emph{for any two distinct idempotents $e_i$ and $e_j$,
  $(e_i+e_j)A(e_i+e_j)$ is isomorphic to $k\llbracket Z_2\rrbracket$.}

Step 8: Consider the case $n=3$ (thus $1_A=e_1+e_2+e_3$) and put
$R_i=e_iAe_i$, $M_i=e_{i+1}Ae_i$, $N_i=e_{i-1}Ae_i$. In this step the
indices are taken modulo $3$. By the previous step we have the
relations
\[ R_{i+1}M_i=M_iR_i,\quad R_{i-1}N_i=N_iR_i,\quad
N_{i+1}M_i=R_i,\quad M_{i-1}N_i=R_i.
\]
Now $M_{i+1}M_i\neq 0$, since $N_{i+2}M_{i+1}M_i=R_{i+1}M_i\neq0$, for
instance, and therefore
\[ M_{i+1}M_i=N_iR_i^{r(i)}=R_{i-1}^{r(i)}N_i
\]
for some uniquely determined integer $r(i)\geq0$. Calculating in two
ways:
\[ M_{i+2}(M_{i+1}M_i)=M_{i+2}N_iR_i^{r(i)}=R_i^{r(i)+1},
\]
\[ (M_{i+2}M_{i+1})M_i=R_i^{r(i+1)}N_{i+1}M_i=R_i^{r(i+1)+1},
\]
we see that $r(i)$ is independent of $i$; denote this integer again by
$r$. Similarly, we have
\[ N_{i-1}N_i=M_iR_i^s=R_{i+1}^sM_i
\]
for all $i$ and some uniquely determined integer $s\geq0$. Calculating
again in two ways:
\[ N_{i+2}(M_{i+1}M_i)=N_{i+2}N_iR_i^r=M_iR_i^sR_i^r=M_iR_i^{r+s},\]
\[ (N_{i+2}M_{i+1})M_i=R_{i+1}M_i=M_iR_i,\]
we see that $r+s=1$ and hence that $(r,s)$ equals $(1,0)$ or
$(0,1)$. This means that $A$ is isomorphic to $k\llbracket Z_3\rrbracket$.

Step 9: In the general case, taking into account steps~3 and 4 and
applying steps~6 and 8, one immediately sees that the quiver of $A$ is
a cycle. Since $A$ is infinite dimensional, there cannot be any
relation.
\end{proof}

\end{document}